\documentclass [12pt]{article}
\usepackage{amssymb,amsmath,amsthm}
\usepackage{color}

\def\Im{{\rm Im}\,}

\renewcommand{\span}{span}

\usepackage[cp1250]{inputenc}
\usepackage[active]{srcltx}

\newtheorem{thm}{Theorem}[section]
\newtheorem{prop}[thm]{Proposition}
\newtheorem{lem}[thm]{Lemma}
 \newtheorem{rem}[thm]{Remark}
 \newtheorem{cor}[thm]{Corollary}
\newtheorem{defn}[thm]{Definition}

\DeclareMathOperator{\supp}{supp}
\DeclareMathOperator{\ext}{ext}
\newcommand{\C}{\mathbb{C}}
\newcommand{\R}{\mathbb{R}}

\newcommand{\D}{\mathbb{D}}

\newcommand{\la}{\lambda}

\long\def\comment#1{{}}
\def\ph{\phantom{-}}
\def\p{\ph}
 \title{Indeterminate Jacobi operators}
 \author{Christian Berg and Ryszard Szwarc}

\begin{document}

 \maketitle

\begin{abstract} We consider the Jacobi operator $(T,D(T))$ associated with an indeterminate Hamburger moment problem, i.e., the operator in $\ell^2$ defined as the closure of the Jacobi matrix acting on the subspace of complex sequences with only finitely many non-zero terms. It is well-known that it is symmmetric with deficiency indices $(1,1)$. For a complex number $z$ let $\mathfrak{p}_z, \mathfrak{q}_z$ denote the square summable sequences $(p_n(z))$ and $(q_n(z))$    
 corresponding to the orthonormal polynomials $p_n$ and  polynomials $q_n$ of the second kind. We determine whether linear combinations of $\mathfrak{p}_u, \mathfrak{p}_v,\mathfrak{q}_u,\mathfrak{q}_v$ for $u,v\in\C$ belong to $D(T)$ or to the domain of the self-adjoint extensions of $T$ in $\ell^2$.
The results depend on the four Nevanlinna functions of two variables associated with the moment problem. We also show that $D(T)$ is the common range of an explicitly constructed family of bounded operators on $\ell^2$.  
\end{abstract}



\section{Introduction} We shall consider the Jacobi matrix $J$ associated with a moment sequence $s=(s_n)_{n\ge 0}$ of the form
\begin{equation}\label{eq:mom}
s_n=\int x^n\,d\mu(x),\quad n=0,1,\ldots,
\end{equation}
where $\mu$ is a positive measure on $\R$  with infinite support and moments of every order. It is a tridiagonal matrix of the  form
\begin{equation}\label{eq:Jac}
J=\begin{pmatrix}
b_0 & a_0 & 0 & \hdots\\
a_0 & b_1 & a_1 & \hdots\\
0 & a_1 & b_2 & \hdots\\
\vdots &\vdots & \vdots & \ddots
\end{pmatrix},
\end{equation}
where $a_n>0, b_n\in\R, n\ge 0$  are given by the three term recurrence relation
\begin{equation*}\label{eq:3term}
xp_n(x)=a_np_{n+1}(x)+b_np_n(x)+a_{n-1}p_{n-1}(x), n\ge 0, \quad a_{-1}:=0.
\end{equation*}
Here $(p_n)_{n\ge 0}$ is the sequence of orthonormal polynomials associated with $\mu$, hence satisfying
\begin{equation*}\label{eq:op}
\int p_n(x)p_m(x)\,d\mu(x)=\delta_{n,m},
\end{equation*}
and $p_n$ is a real polynomial of degree $n$ with positive leading coefficient. In this paper   we follow the terminology of \cite{Sch}. Basic results about the classical moment problem can also be found in \cite{Ak} and \cite{S}. Recent  results about indeterminate moment problems can be found in \cite{B:S1}, \cite{B:S2} \cite{B:S3}, \cite{B:S4}.

 It is  clear that the proportional measures $\la\mu, \la>0$ lead to the same Jacobi matrix $J$, and the well-known Theorem of Favard (see \cite[Theorem 5.14]{Sch}) states that any matrix of the form \eqref{eq:Jac} with $a_n>0, b_n\in\R$ comes from a unique  moment sequence  $(s_n)$ as above, normalized such that $s_0=1$. In the following we shall always assume that this normalization holds, and consequently the solutions $\mu$ of \eqref{eq:mom} are probability measures  and $p_0=1$.

The Jacobi matrix acts as a symmetric operator  in the Hilbert space $\ell^2$ of square summable complex sequences. Its domain $\mathcal F$ consists of the complex sequences $(c_n)_{n\ge 0}$ with only finitely many non-zero terms, and the action is multiplication of the matrix $J$ by $c\in\mathcal F$ considered as a column, i.e.,
\begin{equation}\label{eq:J} 
(Jc)_n:=a_{n-1}c_{n-1}+b_n c_n+a_n c_{n+1},\quad n\ge 0.
\end{equation}
Denoting $(e_n)_{n\ge 0}$ the standard orthonormal basis of $\ell^2$, we have
\begin{equation*}\label{eq:F}
\mathcal F=\span\{e_n|n\ge 0\}.
\end{equation*}

\begin{defn} The Jacobi operator associated with $J$ is by definition the closure $(T,D(T))$ of the symmetric operator $(J,\mathcal F)$.
\end{defn}

It is a classical fact that $(T,D(T))$ is a closed symmetric operator, and its deficiency indices are either $(0,0)$  or $(1,1)$. These cases occur precisely if the  moment sequence \eqref{eq:mom} is {\it determinate} or {\it indeterminate}, i.e., there is exactly one or several solutions $\mu$ satisfying \eqref{eq:mom}.   

 By definition $D(T)$ consists of those $c\in\ell^2$ for which there exists a sequence
 $(c^{(k)})\in\mathcal F$ such that $\lim_{k\to\infty}c^{(k)}=c$ and $(Jc^{(k)})$ is a convergent sequence in $\ell^2$. For  such $c$ we have $Tc=\lim_{k\to\infty} Jc^{(k)}$, and this limit is independent of the choice of approximating sequence $(c^{(k)})$. 
 
 Clearly, $D(T)$ is closed under complex conjugation and 
 $$
 T\overline{c}=\overline{Tc},\quad c\in D(T).
 $$
 
 The purpose of the present paper is to study the Jacobi operator $(T,D(T))$ as well as its self-adjoint extensions $(T_t, D(T_t)), t\in\R^*:=\R\cup\{\infty\}$ in the indeterminate case. We shall in particular give some families of sequences $c\in\ell^2$ which belong to $D(T)$, see Theorem~\ref{thm:pqmain}--Theorem~\ref{thm:pos}. 
 
 Section 2 is devoted to the proof of Theorem~\ref{thm:pqmain} after  a presentation of the deficiency spaces of $(T, D(T))$. The self-adjoint extensions of $(T, D(T))$ as well as their corresponding N-extremal solutions to \eqref{eq:mom}, cf. \eqref{eq:Next}, are introduced in Section 3. 
 
 In Theorem~\ref{thm:pro1}, Theorem~\ref{thm:D(T_t)} and Theorem~\ref{thm:pro2} we describe vectors belonging to $D(T_t)\setminus D(T)$. Like the results in Theorem~\ref{thm:pqmain}--Theorem~\ref{thm:pos}, they depend on the Nevanlinna functions of two variables defined in \eqref{eq:A},  \eqref{eq:B}, \eqref{eq:C}, \eqref{eq:D}.
 
 In Section 4 we construct for each $z_0\in\C$ a bounded operator $\Xi_{z_0}$ in $\ell^2$ with range $D(T)$. The restriction of $\Xi_{z_0}$ to $(T-z_0I)(D(T))$ is a bijection onto $D(T)$ equal to $(T-z_0I)^{-1}$, see Theorem~\ref{thm:para}. It is based
  on a study of the function space $\mathcal E$ defined in \eqref{eq:E}, and known to be a de Branges space of entire functions by \cite[Theorem 23]{Br}. We prove in particular Theorem~\ref{thm:deB}, showing that $\mathcal E$ is stable under the formation of difference quotients.
 
 Various technical  results about the Nevanlinna functions are given in Section 5.
 
After this summary of the content of the present paper, we recall that the adjoint operator $(T^*,D(T^*))$ is the maximal operator associated with $J$, cf. \cite[Proposition 6.5]{Sch}. In fact, the matrix product of $J$  and any column vector $c$ makes sense, cf. \eqref{eq:J}, and $D(T^*)$ consists of those $c\in\ell^2$ for which the product $Jc$ belongs to $\ell^2$. For $c\in D(T^*)$ we have $T^*c=Jc$.
 
In the determinate case with a unique solution $\mu$ of \eqref{eq:mom}, the Jacobi operator is self-adjoint and $(p_n)$ is an orthonormal basis of $L^2(\mu)$. The self-adjoint operator of multiplication $M_\mu$ in $L^2(\mu)$ given by 
$$
D(M_\mu)=\{f\in L^2(\mu)\mid xf(x)\in L^2(\mu)\},\quad M_\mu f(x)=xf(x)
$$
is unitarily equivalent with $(T, D(T))$ via the unitary operator $U:\ell^2\to L^2(\mu)$ given by $U(e_n)=p_n, n\ge 0$. We shall not study the determinate case in this paper, but concentrate on the  indeterminate case, where it is known that  the set of solutions $\mu$ to \eqref{eq:mom} is an infinite convex set $V$. The polynomials of the second kind $(q_n)$ are given as
\begin{equation*}
q_n(z)=\int \frac{p_n(z)-p_n(x)}{z-x}\,d\mu(x), \quad z\in \C,
\end{equation*}
where $\mu\in V$ is arbitrary.

We define  and recall
\begin{equation}\label{eq:frak}
 \mathfrak{p}_z:=(p_n(z)), \mathfrak{q}_z:=(q_n(z))\in\ell^2,\quad z\in\C,
 \end{equation}
 where we have followed the terminology of \cite{Sch}. It is known that
   $||\mathfrak{p}_z||$ and $||\mathfrak{q}_z||$ are positive continuous functions on $\C$.  It is therefore possible 
for $c\in\ell^2$ to define entire functions $F_c, G_c$ as
\begin{equation}\label{eq:FGc}
F_c(z)=\sum_{n=0}^\infty c_np_n(z),\quad G_c(z)=\sum_{n=0}^\infty c_nq_n(z),\quad z\in\C.
\end{equation}

We also have the following four entire functions of two complex variables, called the {\it Nevanlinna functions} of the indeterminate moment problem:
\begin{eqnarray}
A(u,v)&=&(u-v)\sum_{k=0}^\infty q_k(u)q_k(v)\label{eq:A}\\
B(u,v)&=&-1+(u-v)\sum_{k=0}^\infty p_k(u)q_k(v) \label{eq:B}\\
C(u,v)&=&1+(u-v)\sum_{k=0}^\infty q_k(u)p_k(v)\label{eq:C}\\
D(u,v)&=&(u-v)\sum_{k=0}^\infty p_k(u)p_k(v)\label{eq:D},
\end{eqnarray}
see Section 7.1 in \cite{Sch}. The two-variable  functions were introduced in \cite{Bu:Ca} in a slightly different form, which was subsequently  used in \cite{B},\cite{Pe}. An approximation to the two-variable functions was already considered in \cite[p. 123]{Ak}.
If the functions of \cite{Bu:Ca} are marked with a $*$, we have
\begin{eqnarray*}\label{eq:A*-D*}
A^*(u,v)&=&-A(u,v),\; B^*(u,v)=-C(u,v),\\
C^*(u,v)&=&-B(u,v),\; D^*(u,v)=-D(u,v).
\end{eqnarray*} 
 In the following we need several formulas about these functions, see Theorem~\ref{thm:3points} and Corollary~\ref{thm:2to1} in the Appendix, but at this point we just recall that 
\begin{equation}\label{eq:det=1}
A(u,v)D(u,v)-B(u,v)C(u,v)=1,\quad u,v\in\C.
\end{equation}
We define entire functions of one variable by setting the second variable to 0, i.e., 
\begin{equation}\label{eq:A-D}
A(u)=A(u,0),\;B(u)=B(u,0),\;C(u)=C(u,0),\;D(u)=D(u,0),
\end{equation}
and by specialization of \eqref{eq:det=1} we get
\begin{equation}\label{eq:det=1*}
A(u)D(u)-B(u)C(u)=1,\quad u\in\C.
\end{equation}
By Section 6.5 in \cite{Sch} we have
\begin{equation}\label{eq:p,q}
\mathfrak{p}_z, \mathfrak{q}_z \in D(T^*),\quad T^*\mathfrak{p}_z=z\mathfrak{p}_z, 
T^*\mathfrak{q}_z=e_0+z\mathfrak{q}_z,\quad z\in\C.
\end{equation}

Our first main result is the following:

\begin{thm}\label{thm:pqmain} For all $z\in\C$ we have $\mathfrak{p}_z, \mathfrak{q}_z\notin D(T)$.

Let $u,v\in\C$ be given.
\begin{enumerate}
\item[(i)] There exists $\alpha\in\C$ such that $\mathfrak{p}_u+\alpha \mathfrak{p}_v\in D(T)$ if and only if $D(u,v)=0$. In the affirmative case $\alpha$  is uniquely determined as
$\alpha=B(u,v)$. 

\item[(ii)] There exists $\beta\in\C$ such that $\mathfrak{q}_u+\beta \mathfrak{q}_v\in D(T)$ if and only if $A(u,v)=0$. In the affirmative case $\beta$  is uniquely determined as $\beta=-C(u,v)$.

\item[(iii)] There exists $\gamma\in\C$ such that $\mathfrak{p}_u+\gamma \mathfrak{q}_v\in D(T)$ if and only if $B(u,v)=0$. In the affirmative case $\gamma$  is uniquely determined as $\gamma=-D(u,v)$. In particular $\mathfrak{p}_u+\gamma\mathfrak{q}_u\notin D(T)$ for all $u, \gamma\in\C$.
\end{enumerate}
\end{thm}

The proof will be given in Section 2.

We shall next give results about  the  zero-sets of the entire functions $A,\ldots,D$ of two variables.

\begin{thm}\label{thm:zeros} Let $F(u,v)$ denote any of the four functions $A,B,C,D$ on $\C^2$. For  $v\in\C$ define 
\begin{equation}\label{eq:zset}
Z(F)_v:=\{u\in\C \mid F(u,v)=0\}.
\end{equation}
Then $Z(F)_v$ is countably infinite. If $v\in\R$ then $Z(F)_v\subset \R$, and if $v$ is in either the upper or lower half-plane, then $Z(F)_v$ belongs to the same half-plane.
\end{thm}

As a follow up on the two previous theorems we have the following:
\begin{thm}\label{thm:pos} Let  $v\in\R$ be given and consider the set of real zeros $Z(F)_v$ from Theorem~\ref{thm:zeros}.
\begin{enumerate}
\item[(i)] Let $u\in Z(D)_v$ be such that $u<v$ and such that $]u,v[\cap Z(D)_v=\emptyset$. Then $B(u,v)>0$, where  $\mathfrak{p}_u+B(u,v)\mathfrak{p}_v\in D(T)$ according to Theorem~\ref{thm:pqmain}.
\item[(ii)] Let $u\in Z(A)_v$ be such that $u<v$ and such that $]u,v[\cap Z(A)_v=\emptyset$. Then $C(u,v)<0$, where  $\mathfrak{q}_u-C(u,v)\mathfrak{q}_v\in D(T)$ according to Theorem~\ref{thm:pqmain}.
\end{enumerate}
\end{thm}

The proofs of Theorem~\ref{thm:zeros} and Theorem~\ref{thm:pos} will be given in Section 5.

\section{Preliminaires and proof of Theorem~\ref{thm:pqmain}} 

 Fix $z_0\in\C$ in the open upper half-plane and consider the deficiency spaces
\begin{eqnarray*}
\Delta^+(z_0) &=&\ker(T^*-z_0I)=\C \mathfrak{p}_{z_0}\\
\Delta^-(z_0) &=&\ker(T^*-\overline{z_0}I)=\C\mathfrak{p}_{\overline{z_0}},
\end{eqnarray*}
cf. \eqref{eq:p,q}.

We know from \cite[section 80]{A:G} that
\begin{equation}\label{eq:direct}
D(T^*)=D(T)\oplus \Delta^+(z_0) \oplus \Delta^-(z_0),
\end{equation} 
and the sum is direct as indicated by the $\oplus$ signs.

\begin{prop}\label{thm:pq} For any $\la\in \C$ we have the decomposition from \eqref{eq:direct} 
\begin{equation}\label{eq:p}
\mathfrak{p}_{\la}=s_\la +s_\la^+\mathfrak{p}_{z_0}+s_\la^-\mathfrak{p}_{\overline{z_0}},
\end{equation}
where $s_\la\in D(T)$ and
\begin{equation}\label{eq:s+-}
s_\la^+=\frac{D(\la,\overline{z_0})}{2i\Im(z_0)||\mathfrak{p}_{z_0}||^2},\quad 
s_\la^-=-\frac{D(\la, z_0)}{2i\Im(z_0)||\mathfrak{p}_{z_0}||^2}.
\end{equation}
Similarly, we have the decomposition
\begin{equation}\label{eq:q}
\mathfrak{q}_{\la}=r_\la +r_\la^+\mathfrak{p}_{z_0}+r_\la^-\mathfrak{p}_{\overline{z_0}},
\end{equation}
where $r_\la\in D(T)$ and
\begin{equation}\label{eq:r+-}
r_\la^+=\frac{C(\la,\overline{z_0})}{2i\Im(z_0)||\mathfrak{p}_{z_0}||^2},\quad 
r_\la^-=-\frac{C(\la, z_0)}{2i\Im(z_0)||\mathfrak{p}_{z_0}||^2}.
\end{equation}
\end{prop}

\begin{proof} Applying the operator $T^*-\overline{z_0}I$ to the Equation \eqref{eq:p} gives
$$
(T^*-\overline{z_0}I)\mathfrak{p}_{\la}=(T-\overline{z_0}I)s_\la +s_\la^+(z_0-\overline{z_0})\mathfrak{p}_{z_0},
$$
which is the splitting of the left-hand side   according to the orthogonal decomposition
\begin{equation}\label{eq:os}
\ell^2=(T-\overline{z_0}I)(D(T)) \oplus \Delta^+(z_0).
\end{equation}
Therefore, $s_\la^+(z_0-\overline{z_0})\mathfrak{p}_{z_0}$ is the orthogonal projection of 
$(T^*-\overline{z_0}I)\mathfrak{p}_{\la}$ onto  $\Delta^+(z_0)$, and hence
$$
2i\Im(z_0)s_\la^+\mathfrak{p}_{z_0}=\langle(T^*-\overline{z_0}I)\mathfrak{p}_{\la},\mathfrak{p}_{z_0}\rangle
\frac{\mathfrak{p}_{z_0}}{||\mathfrak{p}_{z_0}||^2},
$$
which gives the first formula in \eqref{eq:s+-}. The second formula is obtained similarly by applying the operator $(T^*-z_0I)$ to the Equation \eqref{eq:p}. Notice that $||\mathfrak{p}_{z_0}||=||\mathfrak{p}_{\overline{z_0}}||$.

Applying the operator $T^*-\overline{z_0}I$ to the Equation \eqref{eq:q} gives
$$
(T^*-\overline{z_0}I)\mathfrak{q}_{\la}=(T-\overline{z_0}I)r_\la +r_\la^+(z_0-\overline{z_0})\mathfrak{p}_{z_0},
$$
which is the splitting of the left-hand side   according to the orthogonal decomposition
\eqref{eq:os}.
Therefore, $r_\la^+(z_0-\overline{z_0})\mathfrak{p}_{z_0}$ is the orthogonal projection of 
$(T^*-\overline{z_0}I)\mathfrak{q}_{\la}$ onto  $\Delta^+(z_0)$, and hence
$$
2i\Im(z_0)r_\la^+\mathfrak{p}_{z_0}=\langle(T^*-\overline{z_0}I)\mathfrak{q}_{\la},\mathfrak{p}_{z_0}\rangle
\frac{\mathfrak{p}_{z_0}}{||\mathfrak{p}_{z_0}||^2},
$$
which gives the first formula in \eqref{eq:r+-} because $T^*(\mathfrak{q}_{\la})=\la\mathfrak{q}_{\la}+e_0$ by \eqref{eq:p,q}. The second formula is obtained similarly by applying the operator $(T^*-z_0I)$ to the Equation \eqref{eq:q}.
\end{proof}

\begin{cor}\label{thm:pqnotinDA} For all  $\la\in\C$  we have $\mathfrak{p}_{\la},\mathfrak{q}_{\la}\notin D(T)$. 
\end{cor}

\begin{proof} For $\la =z_0$ we get from \eqref{eq:p} and \eqref{eq:s+-}
$$
s_{z_0}=0,\quad s_{z_0}^+=1,\quad s_{z_0}^-=0,
$$
showing that $\mathfrak{p}_{z_0}\notin D(T)$. The case $\la=\overline{z_0}$ follows because $D(T)$ is closed under complex conjugation. Since $z_0$ in the upper half-plane is arbitrary, the  assertion about $\mathfrak{p}_{\la}$ follows for $\la\in\C\setminus\R$. 

For $\la\in\R$ we note that $s_\la^-=\overline{s_\la^+}\neq 0$ because $z\mapsto D(\la,z)$ has only real zeros, cf. \cite[Theorem 3]{B:C}
or Theorem~\ref{thm:zeros}.

For $\la =z_0$ we get from \eqref{eq:r+-} that
$$
 r_{z_0}^-=\frac{-1}{2i\Im(z_0)||\mathfrak{p}_{z_0}||^2}\neq 0,
$$
showing that $\mathfrak{q}_{z_0}\notin D(T)$ and hence also $\mathfrak{q}_{\overline{z_0}}\notin D(T)$. Since $z_0$ in the upper half-plane is arbitrary, the assertion about $\mathfrak{q}_{\la}$ follows for $\la\in\C\setminus\R$. 

For $\la\in\R$ we note that $r_\la^{-}=\overline{r_\la^+}$, and by \eqref{eq:C1} we have
$$
C(\la,z_0)=D(z_0)[A(\la)-\rho C(\la)]\neq 0,\quad \rho:=B(z_0)/D(z_0).
$$
because $D$ has only real zeros. Furthermore, $\Im\rho>0$ because $B/D$ is a Pick function, cf. Proposition~\ref{thm:Pick}, so also the second factor is non-zero.
\end{proof}

\begin{rem}\label{thm:simple}{\rm Concerning Corollary~\ref{thm:pqnotinDA}, it is clear that $\mathfrak{p}_\la\notin D(T)$ for $\la\notin\R$ because otherwise  $\mathfrak{p}_\la$ would be an eigenvector for $T$ with eigenvalue $\la$, and as $T$ is symmetric, the eigenvalues are real. 

A small modification yields also that $\mathfrak{q}_\la\notin D(T)$ for $\la\notin\R$. 
In fact, otherwise by symmetry of $T$
\begin{equation}\label{eq:sym}
\langle T\mathfrak{q}_\la, \mathfrak{q}_\la\rangle=\langle \mathfrak{q}_\la,T \mathfrak{q}_\la\rangle.
\end{equation}
The left-hand side of \eqref{eq:sym} equals $\langle \la\mathfrak{q}_\la+e_0, \mathfrak{q}_\la\rangle=
\la||\mathfrak{q}_\la||^2$
because $\langle e_0,\mathfrak{q}_\la\rangle=0$.

Similarly, the right-hand side of \eqref{eq:sym} equals $\overline{\la}||\mathfrak{q}_\la||^2$,
and finally $\la$ must be real. We show later that $(T, D(T))$ has no eigenvalues at all, cf. \eqref{eq:reg}.
}
\end{rem}

{\it Proof of Theorem~\ref{thm:pqmain}}.

Corollary~\ref{thm:pqnotinDA} proves the first assertion,  and from this assertion it is clear that there exists at most one number $\alpha$ satisfying $\mathfrak{p}_u+\alpha\mathfrak{p}_v\in D(T)$, and similarly with $\beta, \gamma$.

 Let us now prove assertion (i) of the theorem.

By \eqref{eq:p} we get
$$
\mathfrak{p}_u+\alpha \mathfrak{p}_v=(s_u+\alpha s_v) + (s_u^+ +\alpha s_v^+)\mathfrak{p}_{z_0}+
(s_u^- +\alpha s_v^-)\mathfrak{p}_{\overline{z_0}},
$$
so $\mathfrak{p}_u+\alpha\mathfrak{p}_v\in D(T)$ if and only if 
$$
s_u^+ +\alpha s_v^+=s_u^- +\alpha s_v^-=0,
$$
which by \eqref{eq:s+-} is equivalent to 
\begin{equation}\label{eq:Di}
D(u,\overline{z_0})+\alpha D(v,\overline{z_0})=D(u,z_0)+\alpha D(v,z_0)=0.
\end{equation}
 The determinant of this linear system is 
\begin{equation*}\label{eq:det}
\mathcal D:=D(u,\overline{z_0})D(v,z_0)-D(u,z_0)D(v,\overline{z_0})
\end{equation*}
and using Lemma~\ref{thm:doubledet} with
$$
x=\begin{pmatrix} B(u)\\D(u) \end{pmatrix},\; y=\begin{pmatrix}B(\overline{z_0})\\
D(\overline{z_0})\end{pmatrix},\; z=\overline{y},\;w=\begin{pmatrix} B(v)\\D(v)\end{pmatrix},
$$
we get from Corollary~\ref{thm:2to1} and \eqref{eq:D1}
$$
\mathcal D=\left|\begin{array}{cc}
 B(u)&\;B(v)\\D(u)&\;D(v)\end{array}\right|\left|\begin{array}{cc}
 B(\overline{z_0})&\;B(z_0)\\D(\overline{z_0})&\;D(z_0)\end{array}\right|
=D(u,v)D(\overline{z_0},z_0).
$$
However, $D(\overline{z_0},z_0)=-2\Im(z_0)||\mathfrak{p}_{z_0}||^2\neq 0$ so $\mathcal D=0$ iff $D(u,v)=0$.
Therefore, if $\alpha$ is a solution to \eqref{eq:Di} we have $D(u,v)=0$. Suppose next that  
$D(u,v)=0$. To see that \eqref{eq:Di} has a solution $\alpha$, we notice that $D(v,\overline{z_0})$ and $D(v,z_0)$ cannot both be zero.
In fact, defining $\rho:=B(z_0)/D(z_0)$ we have $\Im(\rho)>0$ because $B/D$ is a Pick function, cf. Proposition~\ref{thm:Pick}, and
by \eqref{eq:D1} $D(v,\overline{z_0})=0$ iff $B(v)=\overline{\rho}D(v)$ while $D(v,z_0)=0$ iff $B(v)=\rho D(v)$, so both equations cannot hold. Here we use that $B(v)=D(v)=0$ is impossible because of \eqref{eq:det=1*}.

If $D(v,z_0)\neq 0$, then $\alpha:=-D(u,z_0)/D(v,z_0)$ satisfies \eqref{eq:Di} because $\mathcal D=0$. Similarly $\alpha:=-D(u,\overline{z_0})/D(v,\overline{z_0})$ satisfies \eqref{eq:Di} if $D(v,\overline{z_0})\neq 0$.

Furthermore, in the case $D(v,z_0)\neq 0$ we get using \eqref{eq:D3} and $D(u,v)=0$ that
$$
D(u,z_0)=D(u,v)C(v,z_0)-B(u,v)D(v,z_0)=-B(u,v)D(v,z_0),
$$
so finally $\alpha=B(u,v)$. The case $D(v,\overline{z_0})\neq 0$ is similar.

\medskip
{\it Proof of (ii)}:

By \eqref{eq:q} we get
$$
\mathfrak{q}_u+\beta \mathfrak{q}_v=(r_u+\beta r_v) + (r_u^+ +\beta r_v^+)\mathfrak{p}_{z_0}+
(r_u^- +\beta r_v^-)\mathfrak{p}_{\overline{z_0}},
$$
so $\mathfrak{q}_u+\beta\mathfrak{q}_v\in D(T)$ if and only if 
$$
r_u^+ +\beta r_v^+=r_u^- +\beta r_v^-=0,
$$
which by \eqref{eq:r+-} is equivalent to 
\begin{equation}\label{eq:Ci}
C(u,\overline{z_0})+\beta C(v,\overline{z_0})=C(u,z_0)+\beta C(v,z_0)=0.
\end{equation}
 The determinant of this linear system is 
\begin{equation*}\label{eq:det1}
{\mathcal D}_1:=C(u,\overline{z_0})C(v,z_0)-C(u,z_0)C(v,\overline{z_0})
\end{equation*}
and using Lemma~\ref{thm:doubledet} with
$$
x=\begin{pmatrix} A(u)\\C(u) \end{pmatrix},\; y=\begin{pmatrix}B(\overline{z_0})\\
D(\overline{z_0})\end{pmatrix},\; z=\overline{y},\;w=\begin{pmatrix} A(v)\\C(v)\end{pmatrix}
$$
we get from Corollary~\ref{thm:2to1} combined with \eqref{eq:A1}, \eqref{eq:C1}, \eqref{eq:D1}
$$
{\mathcal D}_1=\left|\begin{array}{cc}
 A(u)&\;A(v)\\C(u)&\;C(v)\end{array}\right|\left|\begin{array}{cc}
 B(\overline{z_0})&\;B(z_0)\\D(\overline{z_0})&\;D(z_0)\end{array}\right|
=A(u,v)D(\overline{z_0},z_0).
$$
As in case (i) we see that
${\mathcal D}_1=0$ iff $A(u,v)=0$.
Therefore, if $\beta$ is a solution to \eqref{eq:Ci}, we have  $A(u,v)=0$. Suppose next that 
$A(u,v)=0$. To see that \eqref{eq:Ci} has a solution $\beta$, we notice as in (i) that $C(v,\overline{z_0})$ and $C(v,z_0)$ cannot both be zero. For this we use that $A/C$ is a Pick function by Proposition~\ref{thm:Pick}.

If $C(v,z_0)\neq 0$, then $\beta:=-C(u,z_0)/C(v,z_0)$ satisfies \eqref{eq:Ci}. Similarly $\beta:=-C(u,\overline{z_0})/C(v,\overline{z_0})$ satisfies \eqref{eq:Ci} if $C(v,\overline{z_0})\neq 0$.

Furthermore, in the case $C(v,z_0)\neq 0$ we get using \eqref{eq:C3} and $A(u,v)=0$ that
$$
C(u,z_0)=C(u,v)C(v,z_0)-A(u,v)D(v,z_0)=C(u,v)C(v,z_0),
$$
so finally $\beta=-C(u,v)$. The case $C(v,\overline{z_0})\neq 0$ is similar.

\medskip
{\it Proof of (iii)}: By \eqref{eq:p} and \eqref{eq:q} we get
$$
\mathfrak{p}_u+\gamma \mathfrak{q}_v=(s_u+\gamma r_v) + (s_u^+ +\gamma r_v^+)\mathfrak{p}_{z_0}+
(s_u^- +\gamma r_v^-)\mathfrak{p}_{\overline{z_0}},
$$
so $\mathfrak{p}_u+\gamma\mathfrak{q}_v\in D(T)$ if and only if 
$$
s_u^+ +\gamma r_v^+=s_u^- +\gamma r_v^-=0,
$$
which by \eqref{eq:s+-} and \eqref{eq:r+-} is equivalent to 
\begin{equation}\label{eq:DCi}
D(u,\overline{z_0})+\gamma C(v,\overline{z_0})=D(u,z_0)+\gamma C(v,z_0)=0.
\end{equation}

 The determinant of this linear system is 
\begin{equation*}\label{eq:det2}
{\mathcal D}_2:=D(u,\overline{z_0})C(v,z_0)-D(u,z_0)C(v,\overline{z_0})
\end{equation*}
and using Lemma~\ref{thm:doubledet} with
$$
x=\begin{pmatrix} B(u)\\D(u) \end{pmatrix},\; y=\begin{pmatrix}B(\overline{z_0})\\
D(\overline{z_0})\end{pmatrix},\; z=\overline{y},\;w=\begin{pmatrix} A(v)\\C(v)\end{pmatrix}
$$
we get from Corollary~\ref{thm:2to1} combined with \eqref{eq:B1}, \eqref{eq:C1}, \eqref{eq:D1}
$$
{\mathcal D}_2=\left|\begin{array}{cc}
 B(u)&\;A(v)\\D(u)&\;C(v)\end{array}\right|\left|\begin{array}{cc}
 B(\overline{z_0})&\;B(z_0)\\D(\overline{z_0})&\;D(z_0)\end{array}\right|
=B(u,v)D(\overline{z_0},z_0).
$$
As in case (i) we see that
${\mathcal D}_2=0$ iff $B(u,v)=0$.
Therefore, if $\gamma$ is a solution to \eqref{eq:DCi}, we have $B(u,v)=0$. Suppose next that $B(u,v)=0$. We see like in (ii) that if $C(v,z_0)\neq 0$, then $\gamma:=-D(u,z_0)/C(v,z_0)$ satisfies \eqref{eq:DCi}, and if $C(v,\overline{z_0})\neq 0$, then $\gamma:=-D(u,\overline{z_0})/C(v,\overline{z_0})$ satisfies \eqref{eq:DCi}.

We finally see that in both cases $\gamma=-D(u,v)$ because of \eqref{eq:D3}.
$\qquad\square$

\begin{rem}\label{thm:trunc}{\rm The case (ii) can be deduced from case (i) by using the observation that
the polynomials $(q_{n+1}(x)/q_1(x))_{n\ge 0}$ are the orthonormal polynomials associated with the truncated Jacobi matrix $J^{(1)}$ obtained from $J$ by removing the first row and column. See \cite[p. 28]{Ak}. We have
$$
J^{(1)}(Sc)=S(Jc)-a_0\langle c, e_0\rangle,\quad c\in\mathcal F,
$$
where $S$ is the bounded shift operator in $\ell^2$ given by $(Sc)_n:=c_{n+1}, n\ge 0$. 
  
If we let $(T^{(1)}, D(T^{(1)}))$ denote the Jacobi operator associated with $J^{(1)}$, one can prove that 
$$
v\in D(T^{(1)})\iff v=Su, u\in D(T)
$$
and 
$$
T^{(1)}(Su)=S(Tu)-a_0\langle u, e_0\rangle,\quad u\in D(T).
$$
}
\end{rem}

\section{Self-adjoint extensions of the Jacobi operator}

As before $(s_n)$ is an indeterminate moment sequence with $s_0=1$. The corresponding Jacobi operator $(T,D(T))$ has deficiency indices $(1,1)$ and the self-adjoint extensions in
$\ell^2$ can be parametrized as the operators $T_t, t\in\R^*=\R\cup \{\infty\}$ with domain  
\begin{equation}\label{eq:domTt}
D(T_t)=D(T)\oplus \C  (\mathfrak{q}_0 + t\mathfrak{p}_0) \;\mbox{for}\;t\in\R,\quad
D(T_\infty)=D(T)\oplus \C \mathfrak{p}_0
\end{equation}
and defined by the restriction of $T^*$ to the domain, cf. \cite[Theorem 6.23]{Sch}.
We recall that $\mathfrak{p}_0, \mathfrak{q}_0$ are defined in \eqref{eq:frak}.

The purpose of this section is to give some results about the domains $D(T_t)$ of the self-adjoint operators $T_t, t\in\R^*$.

For $t\in\R^*$ we define the solutions to the moment problem
\begin{equation}\label{eq:Next}
\mu_t(\cdot):=\langle E_t(\cdot)e_0,e_0\rangle,
\end{equation}
where $E_t(\cdot)$ is the spectral measure of the self-adjoint operator $T_t$.

The measures $\mu_t, t\in\R^*$ are precisely those measures $\mu\in V$  for which the polynomials $\C[x]$ are dense in $L^2(\mu)$ according to a famous theorem of M. Riesz, cf. \cite{Ri}, and they are called N-extremal  in \cite{Ak} and von Neumann solutions in \cite{S}.  They form a closed subset of $\ext(V)$, the set of extreme points of the convex set $V$. However, $\ext(V)$ is known to be a dense subset of $V$.  They are characterized by the formula
\begin{equation}\label{eq:Npar}
\int\frac{d\mu_t(x)}{x-z}=-\frac{A(z)+tC(z)}{B(z)+tD(z)},\quad z\in \C\setminus\R, t\in\R^*,
\end{equation}
where $A,\ldots,D$ are the entire functions given in \eqref{eq:A-D}, cf. \cite[Theorem 7.6]{Sch}. Recall that \eqref{eq:det=1*} holds.

We summarize some of the properties of $\mu_t$, which can be found in \cite{Ak} and \cite{Sch}.

\begin{prop}\label{thm:supportNext}
\begin{enumerate}
\item[(i)] The  solution $\mu_t$ is a discrete measure with support equal to the countable zero set $\Lambda_t$ of the entire function $B(z)+tD(z)$, with the convention that $\Lambda_\infty$ is the zero set of $D$. In particular $\Lambda_t\subset\R$ for $t\in\R^*$.
\item[(ii)] The support of two different N-extremal solutions are disjoint, and each point $x_0\in\R$ belongs to the support of a unique N-extremal measure $\mu_t$, where
$t\in\R^*$ is given  as $t=-B(x_0)/D(x_0)$ if $D(x_0)\neq 0$ and $t=\infty$ if $D(x_0)=0$.
\end{enumerate}
\end{prop}

Let us consider the vector space
\begin{equation}\label{eq:E} 
\mathcal E:=\{F_c(z)=\sum_{n=0}^\infty c_np_n(z) \mid c\in\ell^2\}
\end{equation}
of entire functions, cf. \eqref{eq:FGc}.  It is a Hilbert space under the norm
$$
||F_c||^2=\sum_{n=0}^\infty |c_n|^2=\int |F_c(x)|^2\,d\mu(x),
$$
where $\mu\in V$ can be arbitrary. It is a reproducing kernel Hilbert space of functions with the reproducing kernel
$$
K(u,v):=\sum_{n=0}^\infty p_n(u)p_n(v),\quad u,v\in\C,
$$ 
in the sense that
$$
\int K(u,x)F_c(x)\,d\mu(x)=F_c(u),\quad \mu\in V, u\in\C.
$$
Note that $(p_n)$ is an orthonormal basis of $\mathcal E$, and the mapping $c\mapsto F_c$  
is a unitary operator of the Hilbert space $\ell^2$ onto $\mathcal E$.

For each N-extremal measure $\mu_t$ the mapping $c\mapsto F_c|_{\supp(\mu_t)}$ is a unitary operator of $\ell^2$ onto $L^2(\mu_t)$. The inverse mapping is given by
$$
f\mapsto \left(\langle f, p_n\rangle_{L^2(\mu_t)}\right)_{n\ge 0},\quad f\in L^2(\mu_t),
$$ 
and
\begin{equation}\label{eq:pointwise}
f(x)=\sum_{n=0}^\infty \langle f, p_n\rangle_{L^2(\mu_t)}p_n(x),\quad x\in\supp(\mu_t).
\end{equation}
The series in \eqref{eq:pointwise} converges to $f$ in $L^2(\mu_t)$ and converges also locally uniformly for $x\in\C$, but $f$ is apriori only defined on $\supp(\mu_t)$, so the equality holds pointwise for $x\in\supp(\mu_t)$. The series represents a holomorphic extension of $f$ to all of $\C$.
 
The self-adjoint operator $T_t$ from 
\eqref{eq:domTt} is unitarily equivalent with the multiplication operator $M_{\mu_t}$ on
$L^2(\mu_t)$ given by
\begin{equation*}\label{eq:Mt}
M_{\mu_t}f(x)=xf(x),\quad f\in L^2(\mu_t), x\in \supp(\mu_t).
\end{equation*}

\begin{thm}\label{thm:pro1} Let $\mu_t$ be an N-extremal measure and let $\lambda\in \C\setminus \supp\mu_t$.  
 Then
\begin{equation}\label{eq:imp}w_{\mu_t}(\lambda)\mathfrak{p}_\lambda+\mathfrak{q}_\lambda \in D(T_t),
\end{equation}
 where
\begin{equation}\label{eq:w}
w_{\mu_t}(\lambda):=\int {1\over x-\lambda}\,d\mu_t(x) =-{C(\lambda,x)\over D(\lambda,x)},\quad x\in \supp\mu_t.
\end{equation}
In particular, the ratio $C(\lambda,x)/D(\lambda,x)$
  does not depend on $x\in \supp\mu_t.$
\end{thm}

\begin{proof} Since $\lambda \notin \supp\mu_t$ 
the functions $(x-\lambda)^{-1}$ and $x(x-\lambda)^{-1}$ are bounded on $\supp\mu_t$ and in particular they belong to $L^2(\mu_t)$. Thus $(x-\lambda)^{-1}\in D(M_{\mu_t})$ and we find
$$
\int \frac{p_n(x)}{x-\lambda}\,d\mu_t(x)=w_{\mu_t}(\lambda)p_n(\lambda)+q_n(\lambda),
$$ 
where $w_{\mu_t}(\lambda)$ is given by the first equality of  (\ref{eq:w}). 
Moreover, by \eqref{eq:pointwise}
\begin{equation}\label{eq:x}(x-\lambda)^{-1}=\sum_{n=0}^\infty [w_{\mu_t}(\lambda)p_n(\lambda)+q_n(\lambda)]p_n(x), \quad x\in \supp\mu_t.
\end{equation}
In view of the unitary equivalence of $T_t$ and $M_{\mu_t}$ we get (\ref{eq:imp}).
Multiplying \eqref{eq:x} sidewise by $x-\lambda$ gives 
$$
1=w_{\mu_t}(\lambda)(x-\lambda)\sum_{n=0}^\infty p_n(\lambda)p_n(x)+(x-\lambda) \sum_{n=0}^\infty q_n(\lambda)p_n(x),\quad x\in\supp(\mu_t).
$$
By \eqref{eq:C} and \eqref{eq:D} we therefore get
$$
w_{\mu_t}(\lambda)D(\lambda,x)+C(\lambda,x)=0,\quad x\in \supp\mu_t.
$$
Assume $D(\lambda,x)=0.$ Then $C(\lambda,x)=0,$ but this gives a contradiction to \eqref{eq:det=1}. 
Hence $D(\lambda,x)\neq 0$ and 
$$
w_{\mu_t}(\lambda)=-{C(\lambda, x)\over D(\lambda,x)},\quad x\in \supp\mu_t,
$$
which gives the second part of (\ref{eq:w}).
\end{proof}

\begin{rem}\label{thm:alt} {\rm Using the formulas \eqref{eq:C1} and \eqref{eq:D1} in the last expression in \eqref{eq:w}, we get formula \eqref{eq:Npar} with 
$t=-B(x)/D(x)$ independent of $x\in\supp(\mu_t)$ if $D(x)\neq 0$, and $t=\infty$ if $D(x)=0$.

Note also that $w_{\mu_t}(\lambda)\mathfrak{p}_\lambda+\mathfrak{q}_\lambda \notin D(T)$ by Theorem~\ref{thm:pqmain} (iii) because $B(\la,\la)=-1$.
}
\end{rem}

We know from Theorem~\ref{thm:pqmain} that $\mathfrak{p}_\la, \mathfrak{q}_\la\notin D(T)$ for every $\la\in\C$. We shall now clarify when  $\mathfrak{p}_\la, \mathfrak{q}_\la$  belong to the domain of the self-adjoint extension $T_t$ associated with the N-extremal measure $\mu_t$.
 
\begin{thm}\label{thm:D(T_t)} For $\la\in\C$ and $t\in\R^*$ we have
\begin{equation}\label{eq:pD(T_t)}
\mathfrak{p}_\la\in D(T_t) \iff D(\la,x)=0 \;\forall x\in\supp(\mu_t)\iff \la\in\supp(\mu_t).
\end{equation}
\begin{equation}\label{eq:qD(T_t)}
\mathfrak{q}_\la\in D(T_t) \iff C(\la,x)=0 \;\forall x\in\supp(\mu_t).
\end{equation}
\end{thm}

\begin{proof} We define the entire functions $g_\la, h_\la\in L^2(\mu_t)$ by
$$
g_\la(x):=\sum_{k=0}^\infty p_k(\la)p_k(x),\quad h_\la(x):=\sum_{k=0}^\infty q_k(\la)p_k(x),\quad x\in\C.
$$
If $\mathfrak{p}_\la\in D(T_t)$ then $(T_t-\la I)\mathfrak{p}_\la=0$ by \eqref{eq:p,q}, because $T_t$ is a restriction of $T^*$. Furthermore, by the unitary equivalence between
$T_t$ and the multiplication operator $M_{\mu_t}$ on $L^2(\mu_t)$ we have $xg_\la(x)\in
L^2(\mu_t)$ and $D(\la,x)=(\la-x)g_\la(x)=0$ in $L^2(\mu_t)$. By discreteness of $\mu_t$ we get $D(\la,x)=0$ for all $x\in\supp(\mu_t)$. The last equivalence of \eqref{eq:pD(T_t)} follows from Remark~\ref{thm:Dspec}. On the other hand, it
is easy to see that the last two equivalent conditions of \eqref{eq:pD(T_t)} imply $\mathfrak{p}_\la\in D(T_t)$, because the zero-function $x\mapsto (\la-x)g_\la(x)$ as well as $\la g_\la(x)$ are in $L^2(\mu_t)$, hence also $xg_\la(x)\in L^2(\mu_t)$. Therefore $g_\la\in D(M_{\mu_t})$ and finally 
$\mathfrak{p}_\la\in D(T_t)$. 

If $\mathfrak{q}_\la\in D(T_t)$ then $(T_t-\la I)\mathfrak{q}_\la=e_0$ by \eqref{eq:p,q}, because $T_t$ is a restriction of $T^*$. This shows that $(x-\la)h_\la(x)=1$ in $L^2(\mu_t)$ and hence $C(\la,x)=0$ for all $x\in\supp(\mu_t)$. On the other hand, if $C(\la,x)=0$ for $x\in \supp(\mu_t)$, we conclude that $xh_\la(x)\in L^2(\mu_t)$, hence $\mathfrak{q}_\la\in D(T_t)$. This establishes \eqref{eq:qD(T_t)}.
\end{proof}

We know from Proposition~\ref{thm:supportNext} that $\supp(\mu_t)$ is the the zero set
of the entire function $B(z)+tD(z)$ understood as $D(z)$ if $t=\infty$. Using this we get the following Corollary about $\mathfrak{p}_\la$. We get a similar result about $\mathfrak{q}_\la$ from \eqref{eq:C1}.

\begin{cor}\label{thm:D(T_t)cor} For $t\in\R$ and $t=\infty$ we have
\begin{eqnarray*}
\mathfrak{p}_\la \in  D(T_t) &\iff &  B(\la)+tD(\la)=0,\\
\mathfrak{q}_\la \in  D(T_t) &\iff & A(\la)+tC(\la)=0.\\
\mathfrak{p}_\la \in  D(T_\infty) &\iff &  D(\la)=0,\\
\mathfrak{q}_\la \in  D(T_\infty) &\iff & C(\la)=0.
\end{eqnarray*}
In particular $\mathfrak{p}_\la$ and $\mathfrak{q}_\la$ only belong to $D(T_t)$ if 
$\la\in\R$, and for $\la\in\R$ they belong to a unique $D(T_t)$. Furthermore, they never belong to the same domain $D(T_t)$. 
\end{cor}

\begin{rem}\label{thm:notin}{\rm Since $D(T)\subset D(T_t)$ for all $t\in\R^*$, it is clear that Corollary~\ref{thm:D(T_t)cor} implies that $\mathfrak{p}_\la, \mathfrak{q}_\la\notin D(T)$ as stated in Corollary~\ref{thm:pqnotinDA}.
}
\end{rem}

We also have a kind of converse to Theorem~\ref{thm:pro1}.

\begin{thm}\label{thm:pro2} Assume that $\la,\tau\in\C$ are such that $\tau\mathfrak{p}_\la +\mathfrak{q}_\la\in D(T_t)$ for some $t\in\R^*$. Then $\la\notin \supp(\mu_t)$ and 
$\tau=w_{\mu_t}(\la)$ given by \eqref{eq:w}.
\end{thm}

\begin{proof} Assume that $\la\in\supp(\mu_t)$. By Theorem~\ref{thm:D(T_t)} we know that $\mathfrak{p}_\la\in D(T_t)$ and hence $\mathfrak{q}_\la\in D(T_t)$, contradicting Corollary~\ref{thm:D(T_t)cor}.

Having established $\la\notin\supp(\mu_t)$, we get by Theorem~\ref{thm:pro1} that $(\tau-w_{\mu_t}(\la))\mathfrak{p}_\la\in D(T_t)$, but since
 $\mathfrak{p}_\la\notin D(T_t)$, we get 
$\tau-w_{\mu_t}(\la)=0$.
\end{proof}

 \begin{thm}\label{thm:DT} Let $t\in\R^*$ and $\la\in \C\setminus \supp(\mu_t)$ be given.
 Then there exists a unique pair $(s,c)\in D(T)\times \C$ depending on $t,\la$ such that 
\begin{equation*}
 w_{\mu_t}(\lambda)\mathfrak{p}_\lambda+\mathfrak{q}_\lambda
 =\left\{\begin{array}{ll}
 s +c(\mathfrak{q}_0+t\mathfrak{p}_0), &   t\in\R\\
  s +c\mathfrak{p}_0, & t=\infty.
  \end{array}
  \right.
 \end{equation*}
 We have
 \begin{equation*}
 c=\left\{\begin{array}{ll}
 -1/(B(\la)+tD(\la)), & t\in\R \\
  -1/D(\la), & t=\infty,
    \end{array}
  \right.
  \end{equation*}
   and $s$ is given by inserting the value of $c$.
 \end{thm}
  
  \begin{proof} Recall that $\la\in\C\setminus \supp(\mu_t)$ if and only if $B(\la)+tD(\la)\neq 0$ when $t\in\R$, and that $\la\in\C\setminus \supp(\mu_\infty)$ if and only if
 $D(\la)\neq 0$. The existence and uniqueness of $(s,c)$ follow from Theorem~\ref{thm:pro1} and formula \eqref{eq:domTt}.
 
 In case $t\in\R$ we have
 $$
 w_{\mu_t}(\lambda)\mathfrak{p}_\lambda+\mathfrak{q}_\lambda-c(\mathfrak{q}_0+t\mathfrak{p}_0)\in D(T), 
 $$
 and fixing $z_0$ in the open upper half-plane we have by Proposition 2.1
  $$
 w_{\mu_t}(\la)s_\la^{+} +r_\la^{+} - c(r_0^{+}+ts_0^{+})= w_{\mu_t}(\la)s_\la^{-} +r_\la^{-} - c(r_0^{-}+ ts_0^{-})=0,
 $$
 or equivalently by \eqref{eq:s+-} and \eqref{eq:r+-}
\begin{equation}\label{eq:overlinez_0}
 w_{\mu_t}(\la) D(\la,\overline{z_0})+C(\la,\overline{z_0})-c(C(0,\overline{z_0})+tD(0,\overline{z_0}))=0,
\end{equation}
and
\begin{equation}\label{eq:z_0}
 w_{\mu_t}(\la) D(\la,z_0)+C(\la,z_0)-c(C(0,z_0)+tD(0,z_0))=0.
\end{equation}

From \eqref{eq:z_0} and \eqref{eq:w} we get for $x\in\supp(\mu_t)$
\begin{eqnarray*}
c &=& -\frac{w_{\mu_t}(\la) D(\la,z_0)+C(\la,z_0)}{B(z_0)+tD(z_0)}\\
 &=&  \frac{C(\la,x)D(\la,z_0)-D(\la,x)C(\la,z_0)}{(B(z_0)+tD(z_0))D(\la,x)}\\
 &=&-\frac{D(z_0,\la)C(\la,x)-B(z_0,\la)D(\la,x)}{(B(z_0)+tD(z_0))D(\la,x)}\\
 &=&-\frac{D(z_0,x)}{(B(z_0)+tD(z_0))D(\la,x)}=-\frac{1}{B(\la)+tD(\la)}.
  \end{eqnarray*}
 Here we have first used \eqref{eq:D3} and next used \eqref{eq:D1}  twice. Finally we recall that $t=-B(x)/D(x)$ for $x\in\supp(\mu_t)$, cf. Proposition~\ref{thm:supportNext}. 
 
 Note that \eqref{eq:overlinez_0} leads to the same expression for $c$.
 
 The case $t=\infty$ is treated in the same way. 
 \end{proof}

\section{Parametrizations of the domain of the Jacobi operator}
The Jacobi operator $(T,D(T))$ in the indeterminate case is regular in the sense of \cite[p. 20]{G:G}, i.e., for any $z\in\C$ there exists $d(z)>0$ such that
\begin{equation}\label{eq:reg}
||(T-zI)c||\ge d(z)||c||,\quad c\in D(T).
\end{equation}
For $z\in\C\setminus\R$ this is true with $d(z)=|\Im(z)|$, and for $z\in\R$ let $t_0\in\R^*$ be such that $z\in\supp(\mu_{t_0})$. For $t\in\R^*\setminus\{t_0\}$ the distance 
$$
d_t(z):=\min\{|z-x|\mid x\in \supp(\mu_t)\}>0,
$$ 
can be used in \eqref{eq:reg}, since we have
$$
||(T-zI)c||^2= \int \left|(x-z)F_c(x)\right|^2\,d\mu_t(x)\ge d_t(z)^2||c||^2,
$$
where $F_c$ is given in \eqref{eq:FGc}.

We have the orthogonal decomposition in closed subspaces
\begin{equation}\label{eq:regor}
\ell^2=(T-zI)(D(T))\oplus \C\mathfrak{p}_{\overline{z}},\quad z\in\C.
\end{equation}

The operator $(T,D(T))$ has no eigenvalues, has empty continuous spectrum, and the spectrum $\sigma(T)=\C$ is equal to the residual  spectrum, cf. \cite[p.209]{Y}.
 
 For $z_0\in\C$ we have the orthogonal expansion
\begin{equation}\label{eq:1}
\frac{p_n(z)-p_n(z_0)}{z-z_0}=\sum_{k=0}^{n-1} a_{n,k}(z_0)p_k(z), \quad z\in\C
\end{equation}
of the polynomial $(p_n(z)-p_n(z_0)/(z-z_0)$ of degree $n-1$, and it is easy to see that
\begin{equation}\label{eq:2}
a_{n,k}(z_0)=\int \frac{p_n(x)-p_n(z_0)}{x-z_0}p_k(x)\,d\mu(x)=  q_n(z_0)p_k(z_0)-p_n(z_0)q_k(z_0),
\end{equation}
where $\mu\in V$ is an arbitrary solution to \eqref{eq:mom}, cf. \cite[p. 18]{Ak}.

\begin{lem}\label{thm:ank} The coefficients $a_{n,k}(z_0)$ from \eqref{eq:2}
satisfy
\begin{equation}\label{eq:3}
|a_{n,k}(z_0)|^2\le \left(|p_n(z_0)|^2+|q_n(z_0)|^2\right)\left(|p_k(z_0)|^2+|q_k(z_0)|^2\right).
\end{equation}
Therefore
\begin{equation}\label{eq:3bis}
\sum_{n=k+1}^\infty |a_{n,k}(z_0)|^2\le (||\mathfrak{p}_{z_0}||^2+||\mathfrak{q}_{z_0}||^2)(|p_k(z_0)|^2+|q_k(z_0)|^2).
\end{equation}
Furthermore,
\begin{equation}\label{eq:dif}
\sum_{n=0}^\infty\left|\frac{p_n(z)-p_n(z_0)}{z-z_0}\right|^2\le ||\mathfrak{p}_z||^2\left(||\mathfrak{p}_{z_0}||^2+||\mathfrak{q}_{z_0}||^2\right)^2.
\end{equation}
In particular, for $z\to z_0$
\begin{equation*}\label{eq:diff}
\sum_{n=0}^\infty|p_n'(z_0)|^2\le ||\mathfrak{p}_{z_0}||^2\left(||\mathfrak{p}_{z_0}||^2+||\mathfrak{q}_{z_0}||^2\right)^2.
\end{equation*}
 \end{lem} 
 
 \begin{proof} Formula \eqref{eq:3} is a consequence of the Cauchy-Schwarz inequality.

From \eqref{eq:1} and \eqref{eq:3} we get
\begin{eqnarray*}
\left|\frac{p_n(z)-p_n(z_0)}{z-z_0}\right|^2 &\le& \sum_{k=0}^{n-1} |a_{n,k}(z_0)|^2\sum_{k=0}^{n-1} |p_k(z)|^2\\
&\le& ||\mathfrak{p}_z||^2\sum_{k=0}^{n-1} \left(|p_n(z_0)|^2+|q_n(z_0)|^2\right) \left(|p_k(z_0)|^2+|q_k(z_0)|^2\right),
\end{eqnarray*} 
and finally
\begin{eqnarray*}
\lefteqn{\sum_{n=0}^\infty \left|\frac{p_n(z)-p_n(z_0)}{z-z_0}\right|^2 }\\
&\le & ||\mathfrak{p}_z||^2
\sum_{k=0}^\infty \left[\left(|p_k(z_0)|^2+|q_k(z_0)|^2\right) \sum_{n=k+1}^\infty
 \left(|p_n(z_0)|^2+|q_n(z_0)|^2\right)\right]\\
&\le& ||\mathfrak{p}_z||^2 \left(||\mathfrak{p}_{z_0}||^2+\mathfrak{p}_{z_0}||^2\right)^2,
\end{eqnarray*}
which yields \eqref{eq:dif}.
\end{proof} 

We shall now show that the Hilbert space $\mathcal E=\{F_c(z)\}$ defined in \eqref{eq:E} is stable under difference quotients:
  
\begin{thm}\label{thm:deB} For $c\in\ell^2$  and $z_0\in\C$ there exists $\xi(c,z_0)\in\ell^2$ such that
\begin{equation}\label{eq:Fcq}
\frac{F_c(z)-F_c(z_0)}{z-z_0}=F_{\xi(c,z_0)}(z) \in\mathcal E, 
\end{equation}
and the coordinates of $\xi(c,z_0)$ are defined by 
\begin{equation}\label{eq:xi}
\xi_k(c,z_0)=\sum_{n=k+1}^\infty c_n a_{n,k}(z_0),\quad k\ge 0.
\end{equation}
Furthermore,
\begin{equation}\label{eq:normxi}
||\xi(c,z_0)|| \le ||c||\left(||\mathfrak{p}_{z_0}||^2+||\mathfrak{q}_{z_0}||^2\right).
\end{equation}
\end{thm}

\begin{proof} The series in \eqref{eq:xi} is absolutely convergent being the product of two $\ell^2$ sequences. Furthermore, by the Cauchy-Schwarz inequality and \eqref{eq:3} we get
$$
|\xi_k(c,z_0)|^2\le ||c||^2 (||\mathfrak{p}_{z_0}||^2+||\mathfrak{q}_{z_0}||^2)\left(|p_k(z_0)|^2+|q_k(z_0)|^2\right),
$$
and therefore $(\xi_k(c,z_0))\in\ell^2$ and \eqref{eq:normxi} holds.

We next find
$$
\frac{F_c(z)-F_c(z_0)}{z-z_0}=\sum_{n=0}^\infty c_n\frac{p_n(z)-p_n(z_0)}{z-z_0},\quad z\neq z_0.
$$ 

 Inserting the expression
\eqref{eq:1} on the right-hand side, we get for $z\neq z_0$
\begin{eqnarray*}
\frac{F_c(z)-F_c(z_0)}{z-z_0}&=&\sum_{n=1}^\infty c_n\sum_{k=0}^{n-1} a_{n,k}(z_0)p_k(z)\\
&=&\sum_{k=0}^\infty p_k(z) \sum_{n=k+1}^\infty c_n a_{n,k}(z_0)\\
&=&\sum_{k=0}^\infty\xi_k(c,z_0) p_k(z),
\end{eqnarray*}
where the rearrangement is possible due to absolute convergence:
\begin{eqnarray*}
\lefteqn{\left|\frac{F_c(z)-F_c(z_0)}{z-z_0}\right| }\\
&\le& \sum_{n=1}^\infty |c_n|\sum_{k=0}^{n-1} |a_{n,k}(z_0)|\,|p_k(z)|\\
&=&\sum_{k=0}^\infty |p_k(z)| \sum_{n=k+1}^\infty |c_n|\,|a_{n,k}(z_0)| \\
&\le& ||c||\sum_{k=0}^\infty |p_k(z)|\left[\sum_{n=k+1}^\infty |a_{n,k}(z_0)|^2\right]^{1/2}\\
&\le& ||c||\sum_{k=0}^\infty |p_k(z)|\left[\left(||\mathfrak{p}_{z_0}||^2+||\mathfrak{q}_{z_0}||^2\right)\left(|p_k(z_0)|^2+|q_k(z_0)|^2\right)\right]^{1/2}\\
&\le& ||c||\left(||\mathfrak{p}_{z_0}||^2+||\mathfrak{q}_{z_0}||^2\right)^{1/2}||\mathfrak{p}_z||\left(\sum_{k=0}^\infty\left(|p_k(z_0)|^2+|q_k(z_0)|^2\right)\right)^{1/2}\\
&=&||c||\,||\mathfrak{p}_z||\left(||\mathfrak{p}_{z_0}||^2+||\mathfrak{q}_{z_0}||^2\right),
\end{eqnarray*}
where we have used \eqref{eq:3bis}. 

It is now clear that the entire functions
$z\mapsto (F_c(z)-F_c(z_0))/(z-z_0)$, with value $F'_c(z_0)$ for $z=z_0$, and $F_{\xi(c,z_0)}(z)$ agree.
\end{proof}

\begin{thm}\label{thm:para} Let $\Xi_{z_0}$ denote the bounded operator in $\ell^2$ defined by
\begin{equation}\label{eq:para}
\Xi_{z_0}(c):=\xi(c,z_0),\quad z_0\in\C,\;c\in\ell^2,
\end{equation}
where $\xi(c,z_0)$ is defined in Theorem~\ref{thm:deB}.
We have $\Xi_{z_0}(\ell^2)=D(T)$  and $\ker(\Xi_{z_0})=\C e_0$ for each $z_0\in\C$.

Furthermore, for $z_0\in\C$
\begin{equation}\label{eq:split}
(T-z_0 I)\Xi_{z_0}(c)+F_c(z_0)e_0=c, \quad c\in\ell^2.
\end{equation}
The restriction of $\Xi_{z_0}$ to $(T-z_0I)(D(T))$ is a bijection onto $D(T)$ equal to $(T-z_0I)^{-1}$.
\end{thm}

\begin{proof} Let $U:\ell^2\to \mathcal E$ denote the unitary mapping given by  $U(c)=F_c$. Then
\begin{equation*}
U(Jc)(z)=z\sum_{k=0}^\infty c_kp_k(z),\quad c\in\mathcal F, z\in \C,
\end{equation*}
i.e., $U$ is the intertwining operator between $J$ and the densely defined operator of multiplication with $z$ on $\C[z]\subset \mathcal E$. Therefore 
\begin{equation}\label{eq:unit}
U(Tc)(z)=zF_c(z),\quad c\in D(T), z\in\C.
\end{equation}

For $c\in\ell^2$ and $z_0\in\C$ we have 
$$
c-F_c(z_0)e_0 \perp \mathfrak{p}_{\overline{z_0}},
$$
so by \eqref{eq:regor} $c-F_c(z_0)e_0$ belongs to $(T-z_0 I)(D(T))$. Therefore, there exists a unique vector $v\in D(T)$ such that
\begin{equation}\label{eq:v}
c-F_c(z_0)e_0 =(T-z_0 I)(v),
\end{equation}
and applying $U$ to \eqref{eq:v} we get by \eqref{eq:unit}
$$
F_c(z)-F_c(z_0)=(z-z_0)F_v(z),\quad z\in\C.
$$
Now \eqref{eq:Fcq} shows that $F_v(z)=F_{\xi(c,z_0)}(z)$  for $z\neq z_0$, hence for all $z$, and finally  $v=\xi(c,z_0)$, showing that $\Xi_{z_0}(c)\in D(T)$. Inserting $v=\xi(c,z_0)$ in \eqref{eq:v} yields \eqref{eq:split}. 

For $v\in D(T)$ we define $c=(T-z_0 I)(v)$. Then $F_c(z_0)=0$ as $c\perp \mathfrak{p}_{\overline{z_0}}$ by \eqref{eq:regor}, and then \eqref{eq:split} gives $(T-z_0I)\Xi_{z_0}(c)=c$. By injectivity of $T-z_0I$ we get $v=\Xi_{z_0}(c)=(T-z_0 I)^{-1}(c)$.

It is easy to see that $\xi(e_0,z_0)=0$, hence $\C e_0\subseteq \ker(\Xi_{z_0})$, and from \eqref{eq:split} the converse inclusion follows.
\end{proof}

\begin{rem}\label{thm:noninj}{\rm The operator $\Xi_{z_0}$ defined in \eqref{eq:para} is  seen to satisfy 
$$
\Xi_{z_0}(e_n)=\sum_{k=0}^{n-1} a_{n,k}(z_0) e_k,\quad n\ge 1,
$$
and it follows easily that $\Xi_{z_0}(\mathcal F)=\mathcal F$.

Moreover, since $(\C e_0)^\perp=\{c\in\ell^2 \mid c_0=0\}$, we 
have the following parametrizations of $D(T)$
$$
D(T)=\{\Xi_{z_0}(c) \mid c\in\ell^2, c_0=0\},\quad z_0\in\C.
$$ 
}
\end{rem}

\section{Appendix}
We need the following result about the Nevanlinna functions defined in the Introduction.

\begin{thm}\label{thm:3points} For $u,v,w\in\C$ we have
\begin{eqnarray}
A(u,v)&=& C(u,w)A(w,v)-A(u,w)B(w,v)\label{eq:A3}\\
B(u,v)&=& D(u,w)A(w,v)-B(u,w)B(w,v)\label{eq:B3}\\
C(u,v)&=& C(u,w)C(w,v)-A(u,w)D(w,v)\label{eq:C3}\\
D(u,v)&=& D(u,w)C(w,v)-B(u,w)D(w,v)\label{eq:D3}.
\end{eqnarray}
\end{thm}

From the obvious relations
 $$
 A(u,v)=-A(v,u),\; B(u,v)=-C(v,u),\; D(u,v)=-D(v,u) 
 $$
 and putting $w=0$ in the formulas of Theorem~\ref{thm:3points}, we get the following formulas in terms of the one variable functions \eqref{eq:A-D}:
 
 \begin{cor}\label{thm:2to1} For $u, v\in\C$ we have
 \begin{eqnarray}
A(u,v)&=& \left|\begin{array}{cc}
 A(u)&\;A(v)\\C(u)&\;C(v)\end{array}\right|\label{eq:A1}\\
 B(u,v)&=& \left|\begin{array}{cc}
 B(u)&\;A(v)\\D(u)&\;C(v)\end{array}\right|\label{eq:B1}\\
  C(u,v)&=& \left|\begin{array}{cc}
 A(u)&\;B(v)\\C(u)&\;D(v)\end{array}\right|\label{eq:C1}\\
  D(u,v)&=& \left|\begin{array}{cc}
 B(u)&\;B(v)\\D(u)&\;D(v)\end{array}\right|\label{eq:D1}.
 \end{eqnarray}
\end{cor}

We have not been able to find the formulas of Theorem~\ref{thm:3points} in the literature, so we indicate a proof.  The formulas of Corollary~\ref{thm:2to1} expressing the two variable functions in terms of the one variable functions were, as far as we know,  first  given in \cite{Bu:Ca} and included in \cite[exercise 7.8 (3)]{Sch}.  (Unfortunately there is a misprint in the exercise: $B$ and $C$ are interchanged.)

We begin by introducing polynomial approximations to the Nevanlinna functions.
 
\begin{prop}\cite[Proposition 5.24]{Sch}\label{thm:An-Dn} For $u,v\in\C$ and $n\ge 0$ we have
\begin{eqnarray*}
A_n(u,v)&:=&(u-v)\sum_{k=0}^nq_k(u)q_k(v)=
a_n\left|\begin{array}{cc}
 q_{n+1}(u)&\;q_{n+1}(v)\\q_{n}(u)&\;q_{n}(v)\end{array}\right|\\
B_n(u,v)&:=&-1+(u-v)\sum_{k=0}^n p_k(u)q_k(v)=
a_n\left|\begin{array}{cc}
p_{n+1}(u)&\;q_{n+1}(v)\\p_{n}(u)&\;q_{n}(v)\end{array}\right|\\
C_n(u,v)&:=&1+(u-v)\sum_{k=0}^n q_k(u)p_k(v)=
a_n\left|\begin{array}{cc}
q_{n+1}(u)&\;p_{n+1}(v)\\
q_{n}(u)&\;p_{n}(v)\end{array}\right|\\
D_n(u,v)&:=&(u-v)\sum_{k=0}^np_k(u)p_k(v)=
a_n\left|\begin{array}{cc}
p_{n+1(}u)&\;p_{n+1}(v)\\
p_{n}(u)&\;p_{n}(v)\end{array}\right|.
\end{eqnarray*}
\end{prop}

It is important to notice that
\begin{equation*}\label{eq:det4}
\left|\begin{array}{cc}\ A_n(u,v)&\;B_n(u,v)\\
 C_n(u,v)&\; D_n(u,v)\end{array}\right|=1\textrm{ for }(u,v)\in\mathbb
C^2,
\end{equation*}
cf. \cite[Equation(5.57)]{Sch}.

For later use we introduce the transfer matrix with determinant 1
\begin{equation}\label{eq:transfer1}
h_n(u,v)=\left(\begin{array}{cc} C_n(u,v)&\; A_n(u,v)\\
-D_n(u,v)&\; -B_n(u,v)\end{array}\right),\quad u,v\in\mathbb
C,\;n\ge 0.
\end{equation}
The name transfer matrix is motivated by

\begin{prop}\label{thm:transfer} For $u,v\in\mathbb C,\;n\ge 0$ we have
\begin{equation}\label{eq:transfer2}
\left(\begin{array}{cc}
p_n(u)&\;q_n(u)\\
p_{n+1}(u)&\;q_{n+1}(u)\end{array}\right)h_n(u,v)=\left(\begin{array}{cc}
p_n(v)&\;q_n(v)\\
p_{n+1}(v)&\;q_{n+1}(v)\end{array}\right)
\end{equation}
\end{prop}

\begin{proof} The four formulas of Proposition~\ref{thm:An-Dn} can be expressed
as the matrix equation
$$
h_n(u,v)=a_n
\left(\begin{array}{cc}
q_{n+1}(u)&\;-q_n(u)\\
-p_{n+1}(u)&\;p_{n}(u)\end{array}\right)\left(\begin{array}{cc}
p_n(v)&\;q_n(v)\\
p_{n+1}(v)&\;q_{n+1}(v)\end{array}\right).
$$
However, by \cite[Equation (5.52)]{Sch}
$$
\left(\begin{array}{cc}
q_{n+1}(u)&\;-q_n(u)\\
-p_{n+1}(u)&\;p_{n}(u)\end{array}\right)^{-1}=a_n
\left(\begin{array}{cc}
p_n(u)&\;q_n(u)\\
p_{n+1}(u)&\;q_{n+1}(u)\end{array}\right),
$$
and \eqref{eq:transfer2} follows.\end{proof}

By the uniqueness of a matrix $h_n(u,v)$ satisfying \eqref{eq:transfer2} we get:
\begin{cor}\label{thm:transfercor} For $u,v,w\in\mathbb C,\;n\ge 0$ we
  have
\begin{equation*}\label{eq:transfer3}
h_n(u,w)h_n(w,v)=h_n(u,v),
\end{equation*}
\begin{equation*}\label{eq:transfer4}
h_n(u,v)=h_n(v,u)^{-1}.
\end{equation*}
\end{cor}

{\it Proof of Theorem~\ref{thm:3points}}. Letting $n$ tend to infinity in \eqref{eq:transfer1} we obtain the entire matrix function
$$
h(u,v)=\left(\begin{array}{cc} C(u,v)&\; A(u,v)\\
-D(u,v)&\; -B(u,v)\end{array}\right),\quad u,v\in\mathbb
C,
$$
with determinant 1 satisfying $h(u,w)h(w,v)=h(u,v)$, which is equivalent to the formulas
\eqref{eq:A3},\eqref{eq:B3},\eqref{eq:C3} and \eqref{eq:D3} of the theorem.$\quad\square$

\begin{rem}{\rm The M\"{o}bius transformation $M(u,v):\C^*\to \C^*$ defined by 
\begin{equation*}\label{eq:mob}
M(u,v)(z):= \frac{C(u,v) z+A(u,v)}{-D(u,v)z-B(u,v)},\quad z\in\C^*,
\end{equation*}
maps the Weyl circle $K_v$ onto the Weyl circle $K_u$, where $K_u=\R^*$ if $u\in\R$. For $u\in\C\setminus\R$ the Weyl circle is defined in \cite[Section 7.3]{Sch}.} 
\end{rem}

The following Lemma unifies some calculations:
\begin{lem}\label{thm:doubledet} For vectors $x,y,z,w\in\mathbb C^2$ we have
the following determinant equation
 $$\left|\begin{array}{cc}
|x\;\;y| & |x\;\;z|\\
|w\;\;y| & |w\;\;z|\end{array}\right|=
|x\;\;w||y\;\;z|,
 $$
 where
 $$|x\;\;y|=\left|\begin{array}{cc}
x_1&\;y_1\\
x_2&\;y_2\end{array}\right|
\textrm{ etc.}
 $$
 \end{lem}

\begin{proof} First method: Direct computation.

Second method: Define
 $$M(x,y,z,w)=
|x\;\;y||z\;\;w|+|x\;\;z||w\;\;y|+
|x\;\;w||y\;\;z|,
 $$
 where it should be noticed that $y,z,w$ appear in its cyclic
permutations. Clearly $M$ is a $4$-linear form on $\mathbb C^2$,
and it is alternating, i.e.,\ is zero, if any two arguments agree. An
alternating $4$-linear form on a vector space of dimension $\leq 3$
is identically zero, hence $M\equiv0$.
\end{proof}

In various proofs we need that certain functions are Pick function, i.e., holomorphic functions in the cut plane $\C\setminus\R$ with certain properties, see \cite{Do}.
 
\begin{prop}\label{thm:Pick} The meromorphic functions $B/D$ and $A/C$ are Pick functions, i.e., they map the upper (resp. lower) open half-plane into itself.
\end{prop}

\begin{proof} The result about $B/D$ is in \cite[Proposition 1.3]{B}. The result about $A/C$ can be deduced from the previous result by considering the indeterminate moment problem corresponding to the truncated Jacobi matrix $J^{(1)}$ considered in Remark~\ref{thm:trunc}. There are simple relations between the Nevanlinna functions $\widetilde{A},\ldots,\widetilde{D}$ of the truncated problem and those of the original moment problem, see \cite{Pe}:
\begin{equation*}\label{eq:netr}
A(z)=a_0^{-2}\widetilde{D}(z),\quad C(z)=-b_0 a_0^{-2}\widetilde{D}(z)-\widetilde{B}(z),\quad z\in\C.
\end{equation*}
Therefore $-C/A=b_0+a_0^2(\widetilde{B}/\widetilde{D})$, which shows that $-C/A$ is a Pick function, and so is $A/C$.
\end{proof}

By a famous Theorem of M. Riesz each of the functions $F_c$ defined in \eqref{eq:FGc} are of minimal exponential type meaning that for each $\varepsilon>0$ there exists a constant $C_\varepsilon>0$ such that
$$
|F_c(z)|\le C_\varepsilon e^{\varepsilon |z|},\quad z\in\C.
$$ 
This  follows from the Cauchy-Schwarz inequality because the norm $||\mathfrak{p}_z||$ satisfies the same inequality by \cite[Theorem 2.4.3]{Ak}. Using that the polynomials $q_{n+1}/q_1$ are the orthonormal polynomials for the indeterminate truncated Jacobi matrix $J^{(1)}$, cf. Remark~\ref{thm:trunc}, we also get that the functions $G_c$ from \eqref{eq:FGc} are of minimal exponential type.

We next recall an important property of these functions in case they are not polynomials.

\begin{prop}\label{thm:minexp} For each $c\in \ell^2\setminus\mathcal F$ the functions
$F_c, G_c$ are transcendental  and have a countably infinite set of zeros. 

In particular, for each $v\in\C$ the functions of the variable $u$, $A(u,v), B(u,v)$, $C(u,v), D(u,v)$ have a countably infinite set of zeros. 
\end{prop}

\begin{proof} An entire transcendental function $f$ of minimal exponential type has a countably infinite set of zeros. In fact, the order $\rho$ of $f$ is either strictly less than 1 or equal to 1, and in the latter case the type of $f$ is zero. In the first case the result follows from the Hadamard factorization Theorem, cf. \cite[p.22]{Bo}. In the second case the result follows from a Theorem of Lindel\"{o}f, see \cite[Theorem 2.20.3]{Bo}.  

For $v\in\C$ and $F$ being one of the functions $A,\ldots,D$, we see that $u\mapsto F(u,v)$ is an entire transcendental function of minimal exponential type. 
\end{proof}

{\it Proof of Theorem~\ref{thm:zeros}:} {\bf Case 1:} Let us first consider the case of $D$ with
 $Z(D)_v=\{u\in\C\mid D(u,v)=0\}$ for given $v\in\C$, cf. \eqref{eq:zset}.
 
 If $v\in\R$ then $Z(D)_v\subset\R$ by \cite[Theorem 3]{B:C}, and furthermore $Z(D)_v$ equals the support of the unique N-extremal measure which contains $v$ in the support.
 
 If $v\in\C\setminus\R$, then $D(v)\neq 0$, and using \eqref{eq:D1} we get
 $$
 D(u,v)=D(v)[B(u)-\rho D(u)],\;\rho:=B(v)/D(v),
 $$
 so $u\in Z(D)_v$ iff $B(u)=\rho D(u)$. For such $u$ we must have $D(u)\neq 0$ for otherwise $B(u)=D(u)=0$ contradicting \eqref{eq:det=1*}. This gives $u\in Z(D)_v$ iff
 $B(u)/D(u)=\rho$. Using that $B/D$ is a Pick function by  Proposition~\ref{thm:Pick},  we see that $u,v$ belong to the same half-plane.

{\bf Case 2:} We consider $Z(B)_v$ and  use \eqref{eq:B1}, viz.
$$
B(u,v)=B(u)C(v)-D(u)A(v). 
$$
Let first $v\in\R$. If $C(v)=0$ then $A(v)\neq 0$ by \eqref{eq:det=1*}, so $B(u,v)=0$ iff $D(u)=0$, hence $Z(B)_v\subset\R$. If $C(v)\neq 0$ then
$$
B(u,v)=C(v)[B(u)-\tau D(u)],\;\tau:=A(v)/C(v)\in\R,
$$
so $Z(B)_v$ is the zero set of $B-\tau D$, hence real  by Proposition~\ref{thm:supportNext}. 

Let next $v\in\C\setminus \R$. Then $C(v)\neq 0$, so $u\in Z(B)_v$ iff $B(u)=\tau D(u)$ with $\tau$ as above, but this is only possible if $D(u)\neq 0$ and hence $B(u)/D(u)=\tau$.
Using that both  $A/C$ and $B/D$ are Pick functions, cf. Proposition~\ref{thm:Pick}, we see that $u,v$ belong to the same half-plane.

{\bf Case 3:} We consider $Z(C)_v$ and  use \eqref{eq:C1}, viz.
$$
C(u,v)=A(u)D(v)-C(u)B(v).
$$
By considering the cases $v\in\R$ and $v\in\C\setminus \R$ separately  and factor out $D(v)$ in case it is non-zero, we may
proceed as in case 2.

Finally, the case of $Z(A)_v$  follows from the case 1 by considering the truncated case as in Remark~\ref{thm:trunc}. $\square$

\begin{rem}\label{thm:Dspec} {\rm As noticed in the proof above one has for $v\in\R$:
$$
D(u,v)=0 \iff u\in\supp(\mu),
$$
where $\mu$ is the N-extremal measure such that $v\in\supp(\mu)$.

Compare also with Remark~\ref{thm:alt}.
}
\end{rem}

{\it Proof of Proposition~\ref{thm:pos}}:

{\bf Case (i)}: From the proof of Theorem~\ref{thm:pqmain} (i) we know that 
$B(u,v)=-D(u,z)/D(v,z)$ for all $z\in\C\setminus\R$ since $D(v,z)\neq 0$ for these $z$. By assumption $D(x,v)\neq 0$ for $u<x<v$, so by continuity
$B(u,v)=-D(u,x)/D(v,x)$ for these $x$. We next observe that
\begin{equation}\label{eq:Bhelp}
B(u,v)\frac{v-x}{x-u}=\frac{v-x}{D(v,x)}\frac{D(u,x)}{u-x} \neq 0,\quad u<x<v,
\end{equation} 
and 
$$
\lim_{x\to u^+}\frac{D(u,x)}{u-x}=||\mathfrak{p}_u||^2,\quad
\lim_{x\to v^-}\frac{D(v,x)}{v-x}=||\mathfrak{p}_v||^2,
$$
so the function in \eqref{eq:Bhelp} is positive for $u<x<v$, hence $B(u,v)>0$.

{\bf Case (ii)}: This case is reduced to case (i) for the truncated Jacobi matrix from Remark~\ref{thm:trunc}. If $\widetilde{A},\ldots,\widetilde{D}$ denote the Nevalinna functions of two variables for the truncated case, the following formulas can be found in \cite{Pe}. (The reader is warned that this reference follows the normalization of \cite{Bu:Ca}.)
\begin{equation*}\label{eq:tilde}
\widetilde{B}(u,v)=(v-b_0)A(u,v)-C(u,v),\quad \widetilde{D}(u,v)=a_0^2A(u,v),
\end{equation*}
and since we assume that $A(u,v)=0$, we have $-C(u,v)=\widetilde{B}(u,v)>0$.
$\quad\square$

\noindent
Christian Berg\\
Department of Mathematical Sciences, University of Copenhagen\\
Universitetsparken 5, DK-2100 Copenhagen, Denmark\\
e-mail: {\tt{berg@math.ku.dk}}

\vspace{0.4cm}
\noindent
Ryszard Szwarc\\
Institute of Mathematics, University of Wroc{\l}aw\\
pl.\ Grunwaldzki 2/4, 50-384 Wroc{\l}aw, Poland\\ 
e-mail: {\tt{szwarc2@gmail.com}}

\end{document}